\documentclass[11pt]{article}
\usepackage{amsthm,amsmath,latexsym,amssymb}

\newcommand{\bP}{{\rm |\kern-.15em P}}
\newcommand{\Q}{\kern.3em\rule{.07em}{.65em}\kern-.3em{\rm Q}}
\newcommand{\R}{{\rm I\kern-.15em R}}
\newcommand{\D}{{\rm |\kern-.15em D}}
\newcommand{\h}{{\rm |\kern-.15em H}}
\newcommand{\C}{\kern.3em\rule{.07em}{.65em}\kern-.3em{\rm C}}
\newcommand{\T}{{\rm T\kern-.35em T}}

\theoremstyle{plain}
\newtheorem{theorem}{Theorem}[section]
\newtheorem{lemma}[theorem]{Lemma}

\newtheorem{proposition}[theorem]{Proposition}
\newtheorem{corollary}[theorem]{Corollary}

\theoremstyle{definition}
\newtheorem{definition}[theorem]{Definition}

\theoremstyle{remark}
\newtheorem{remark}[theorem]{Remark}

\begin{document}
\title{Some arithmetical aspects of the two dimensional Jacobian Conjecture}
\author{Ronen Peretz}
 
\maketitle

\begin{abstract}\noindent
We prove that any counter-example to the two dimensional Jacobian Conjecture with integral coefficients and determinant
of its Jacobian matrix equals one, has only paths of asymptotic values that can contain only finitely many points of
the planar integral lattice. The second part of our paper gives some nontrivial applications of the theory of certain
Theta functions to the two dimensional Jacobian Conjecture over the real field, and over the complex field.
\end{abstract}

\section{Introduction}\label{sec0}

The paper demonstrates connections between arithmetic and the two dimensional Jacobian Conjecture. There are two parts in our demonstration
(that are not related to one another). In the first part we show how the theory of diophantine equations have some deep connections to the
Jacobian Conjecture, mostly in two dimensions, but also in $n$ dimensions (for $n\ge 3$ the Jacobian Conjecture is false). Diophantine
equations are closely related to the problem of finding the lattice points of $\mathbb{Z}^2$ on certain planar affine algebraic curves
that are defined with coefficients that belong to finite algebraic extensions of $\mathbb{Z}$. In the second part of our paper we consider
dynamical systems of \'etale maps(in dimension two only) with coefficients which are integers. We are able to conclude some consequences
on the limiting maps of these dynamical systems. That is done using results from the theory of Theta functions. \\
Section \ref{sec1} is mostly a review of theorems of Nguyen Van Chau that cover some deep results, connected to the properties of the solution
sets of certain diophantine equations to properties of Keller maps. We also mention (after Van Chau) a surprising theorem of van de Dries
and Mckena regarding surjective Keller maps. \\
Section \ref{sec2} (still in the first part of our paper) contains a proof of the fact that if $F\in(\mathbb{Z}[X_1,\ldots,X_n])^n$ is \'etale
but not injective, then any path (to infinity) of asymptotic values of $F$ can contain only finitely many $\mathbb{Z}^n$ points. This is an
$n$-dimensional theorem, $N\ge 2$ of interest which is independent of whether the Jacobian Conjecture is valid in dimension $n$ (for $n\ge 3$
it is false). \\
Since we need to have accurate structural results on these asymptotic paths, we bring that structure theorem including its proof (known). \\
Another important tool we need is the Fiber Theorem for polynomial maps, and we recall the work of Kossivi Adjamagbo. \\
Section \ref{sec3} and section \ref{sec4} are dedicated to review some classical results in the theory of Theta functions. This is the
beginning of the second part of our paper. In section \ref{sec5} we use Theta functions in order to state a typical theorem (in fact it is
a theory) that we get using identities on Theta functions. We consider natural dynamical systems of \'etale maps in dimension two. We are
able to identify the natural boundary for the two complex dimensional holomorphic extension of the limiting map. Moreover, we get three
representations of that limiting functions and show the invariance of the \'etale property and of the injectivity property (if the starting
map is injective). We give in parallel the complex setting versus the real setting of these dynamical system.

\section{Some distinguished planar algebraic curves induced by two dimensional \'etale mappings, and the lattice $\mathbb{Z}^2$.}\label{sec1}

Nguyen Van Chau proved some deep theorems that point at relations between the Jacobian Conjecture in any dimension and the size of the
integral solutions of some diophantine systems which are naturally derived from the geometry of \'etale polynomial mapping under consideration. We now review these results of Nguyen Van Chau.

\begin{theorem}\label{A} {\rm \cite{c2022}} \\
Let $K$ be a number field and $\mathcal{O}_K$ the ring of integers in $K$. For every $n>1$ the conjecture {\rm JC}$(K,n)$: Polynomial
mappings $F=(F_1,F_2,\ldots,F_n)\in K[X_1,X_2,\ldots,X_n]^n$ with $JF:=\det\,DF\equiv 1$ {\rm (Keller mappings)} are invertible and hence
have an inverse in $K[X_1,X_2,\ldots,X_n]^n$, is equivalent to the following statement: \\
{\rm DJC}$(K,n)$: For every Keller mapping $F\in\mathcal{O}_K[X_1,X_2,\ldots,X_n]^n$ with \\ $F(0,0,\ldots,0)=(0,0,\dots,0)$ the affine
curve
$$
F_1(X_1,X_2,\ldots,X_n)=F_2(X_1,X_2,\ldots,X_n)=\ldots=F_n(X_1,X_2,\ldots,X_n)\,\,\,\,\,\,\,\,\,\,\,\,\,\,\,(\mathcal{C}_F)
$$
always has non-zero points in $\mathcal{O}_K^n$.
\end{theorem}

\begin{theorem}\label{B} {\rm \cite{c2022}} \\
If the Jacobian Conjecture over $\mathbb{C}$ is {\bf not} true, then for some $n\gg 1$ there exists a Keller map 
$F\in\mathbb{Z}[X_1,X_2,\ldots,X_n]^n$ of the so-called quadratic-cubic linear form over $\mathbb{Z}$,
$$
F_i(X_1,X_2,\ldots,X_n)=X_i+\langle a_i,X\rangle^{d_i},\,\,\,\,\,\,X=(X_1,X_2,\ldots,X_n),
$$
$$
a_i=(a_{i1},a_{i2},\ldots,a_{in})\in\mathbb{Z}^n,\,d_i\in\{2,3\}\,\,\,i=1,2,\ldots,n,
$$
such that the affine curve $(\mathcal{C}_F)$ has {\bf no} non-zero points in $\mathbb{Z}^n$.
\end{theorem}
\noindent
This second theorem implies that the Jacobian Conjecture over $\mathbb{C}$ for all $n>1$ can be reduced to the question whether for every
Keller map $F\in\mathbb{Z}[X_1,X_2,\ldots,X_n]^n$ of quadratic-cubic linear form over $\mathbb{Z}$ the $6$-degree diophantine equation
$F_1(X_1,X_2,\ldots,X_n)^2+\ldots+F_{n-1}(X_1,X_2,\ldots,X_n)^2=0$ always has non-zero integer solutions. We note that while the diophantine
equations to be checked are quite simple, we have to verify the existence of integral solutions for all $n>1$. Also the reduction above as
well as Theorem \ref{B} verifies (or not) the Jacobian Conjecture over $\mathbb{C}$ for all dimensions $n\in\mathbb{Z}_{>0}$. Since the
Jacobian Conjecture over $\mathbb{C}$ is false for $n\ge 3$ we already know that the above tool (Theorem \ref{B}) is limited. \\
We are interested in the current paper only in the the Jacobian Conjecture in dimension two over a field $K$ of characteristic $0$. \\
We would like to replace the diophantine conditions to be checked, in order to get information on the two dimensional Jacobian Conjecture
over $\mathbb{C}$, by stricter conditions of diophantine type. Namely, instead of confining our condition to the existence/non-existence
of integral solutions, we will try to replace those conditions to - checking the existence of an infinite set of integral solutions. Also
we want our new sets of diophantine equations to belong only to the two dimensional case. The price we are willing to pay is that our
diophantine equations will be different than linear lines, or fibers of the coordinates of the polynomial mappings $F(X,Y)=(P(X,Y),Q(X,Y))$
which satisfy the Jacobian Condition:
$$
\frac{\partial P}{\partial X}\cdot\frac{\partial Q}{\partial Y}-\frac{\partial P}{\partial Y}\cdot\frac{\partial Q}{\partial X}\equiv
c\in\mathbb{C}^{\times}.
$$
So we will consider such mappings over $\mathbb{Z}$, i.e. $F\in\mathbb{Z}[X,Y]^2$ but our diophantine equations will not be $P(X,Y)=0$.
Every polynomial surjection of $\mathbb{Z}^n$, which satisfies $JF:=\det\,(\partial F_i/\partial X_j)\equiv 1$ is an automorphism of
$\mathbb{Z}^n$. That setting resembles Keller's original interest which led us later on to express the Jacobian Conjecture over fields of  characteristic zero. The above remarkable reduction to surjectivity was proved in \cite{vm1990}.

\begin{theorem}\label{thm-c2005} {\rm \cite{c2005}}
Let $F=(P,Q):\,\mathbb{C}^2\rightarrow\mathbb{C}^2$ be a polynomial map with integer coefficients, $F=(P,Q)\in\mathbb{Z}[X,Y]^2$, and
$JF\equiv 1$. If the complex plane curve $P=0$ has infinitely many points in $\mathbb{Z}[i]^2$, then $F$ has a polynomial inverse with
integer coefficients.
\end{theorem}
\noindent
In fact the paper \cite{c2005} contains a proof of a stronger result. Namely, if $F=(P,Q)$ is not invertible, then the numbers of points
in $\mathbb{Z}[i]^2$ lying on the coordinate curves $P(X,Y)=c\in\mathbb{Z}[i]$ must be uniformly bounded (in $c$).

\begin{remark}\label{rem1}
van den Dries showed that if $F\in\mathbb{Z}[X,Y]^2$ is a counter-example to the Jacobian Conjecture ($JF\equiv 1$), then $F$ maps 
$\mathbb{Z}[i]^2$ into a narrow neighborhood of its exceptional value set $A_F$. Namely,
$$
{\rm dist}\left(F(p),A_F\right)=\inf_{q\in A_F}|F(p)-q|\le 1,\,\,\,\,\,\forall\,p\in\mathbb{Z}[i]^n.
$$
The exceptional value set $A_F$ of $F$ is the smallest subset $A_F\subseteq\mathbb{C}^n$ such that the restriction 
$F:\,\mathbb{C}^n-F^{-1}(A_F)\rightarrow\mathbb{C}^n-A_F$ is a un-branched covering.
\end{remark}

\begin{proposition}\label{prop-c2005} {\rm \cite{c2005}}
Let $F=(P,Q)\in\mathbb{Z}[X,Y]^2$ with Jacobian $JF\equiv 1$. If $F$ is not invertible, then $\hat{d}(F(p),A_F)\le 1$, 
$\forall\,p\in\mathbb{Z}[i]^2$. Here if $q=(a,b)\in\mathbb{C}^2$ and $V\subseteq\mathbb{C}^2$ we define:
$$
\hat{d}(q,V):=\max\,\left(\min\left\{|a-u||(u,b)\in V\right\},\min\left\{|b-v||(a,v)\in V\right\}\right).
$$ 
\end{proposition}
\noindent
in \cite{c2023} Van Chau mentions the thinness of the image: If $F\in\mathbb{Z}[X_1,X_2,\ldots,X_n]^n$ with $JF\equiv 1$ is not an
invertible polynomial mapping $\mathbb{C}^n\rightarrow\mathbb{C}^n$ and if $E_F$ is the hypersurface of the branched values of $F$
then $F(\mathbb{Z}^n)$ is contained in a narrow neighborhood of $E_F$, namely 
$$
F(\mathbb{Z}^n)\subseteq E_F+\{a\in\mathbb{C}^n|\,|a_i|\le 1,\,i=1,2,\ldots,n\}.
$$
In the case $n=2$, the numbers of integer points on the curves $a\cdot F_1+b\cdot F_2=c$, $a,\,b,\,c\in\mathbb{Z}$ are uniformly
bounded, \cite{c2005}. For any set $\Omega\subseteq\mathbb{Z}^n$ we denote by $N(\Omega,B)$ the number of the points 
$a=(a_1,a_2,\ldots,a_n)\in\Omega$ with $|a_i|\le B,\,i=1,2,\ldots,n$.

\begin{theorem}\label{thm-c2023} {\rm \cite{c2023}}
Suppose $F\in\mathbb{Z}[X_1,X_2,\ldots,X_n]^n$ and $JF\equiv 1$. If $F$ is {\bf not} invertible, then \\
{\rm (i)} For generic lines $l$ in $\mathbb{Z}^n$, the cardinalities $|F^{-1}(l)\cap\mathbb{Z}^n|$ are uniformly bounded. \\
{\rm (ii)} $N(F(\mathbb{Z}^n),B)=O(B^{n-1})$ as $B\rightarrow+\infty$, where the implied constants (in the $O(\cdot)$)
depend on $F$.
\end{theorem}
\noindent
In Theorem \ref{thm-c2023}, by generic lines in $\mathbb{Z}^n$ it is meant that the lines $\{u+t\cdot v|\,t\in\mathbb{C}\}$ for
$(u,v)\in\mathbb{Z}^n\times\mathbb{Z}^n$ lying outside an algebraic surface in $\mathbb{C}^n\times\mathbb{C}^n$ is Zariski closed.
\\
\begin{remark}\label{rem2}
The estimate in Theorem \ref{thm-c2023} is a sharp estimate.
\end{remark}

\begin{corollary}\label{cor-c2023}
Suppose $F\in\mathbb{Z}[X_1,X_2,\ldots,X_n]^n$ and $JF\equiv 1$. The following are equivalent: \\
{\rm (i)} $F$ has an inverse in $\mathbb{Z}[X_1,X_2,\ldots,X_n]^n$. \\
{\rm (ii)} $N(F(\mathbb{Z}^n),B)\gg B^{n-\delta}$ as $B\rightarrow +\infty$ for some $\delta<1$. \\
Here $\alpha\ll\beta$ means that $\alpha\le c\cdot\beta$ for some fixed $c>0$.
\end{corollary}

\section{The finiteness of integral points on any asymptotic path of a $\mathbb{Z}[X,Y]^2$ counterexample to the Jacobian Conjecture
$\mathbb{C}^2\rightarrow\mathbb{C}^2$.}\label{sec2}

As seen in the previous in Section \ref{sec1} there are close connections between mappings $F\in\mathbb{Z}[X_1,X_2,\ldots,X_n]^n$ that 
satisfy the Jacobian condition $JF\equiv 1$, which are not polynomially invertible $\mathbb{C}^n\rightarrow\mathbb{C}^n$, and the
diophantine solvability in two types of diophantine systems: \\
\\
(1) The coordinates of $F$ system:
$$
F_1(X_1,X_2,\ldots,X_n)=F_2(X_1,X_2,\ldots,X_n)=\ldots=F_n(X_1,X_2,\ldots,X_n)\,\,\,\,\,\,\,\,\,\,\,\,\,\,\,(\mathcal{C}_F)
$$
has no non-trivial integral solutions. \\ 
\\
(2) The generic lines systems: \\
For generic lines $l$ in $\mathbb{Z}^n$, the cardinalities $|F^{-1}(l)\cap\mathbb{Z}^n|$ are uniformly bounded. \\
\\
We now restrict our attention to polynomial mappings in dimension two, either $F\in\mathbb{C}[X,Y]^2$ or $F\in\mathbb{Z}[X,Y]^2$ that
satisfy the Jacobian condition
$$
\frac{\partial P}{\partial X}\cdot\frac{\partial Q}{\partial Y}-\frac{\partial P}{\partial Y}\cdot\frac{\partial Q}{\partial X}\equiv 1,
$$
which are not polynomially invertible. We will start by giving a standard construction of all the asymptotic paths of a polynomial mapping 
that generate the full set of its asymptotic values. One purpose of this detailed construction is to examine the relations between the coefficients of the polynomial mapping and the coefficients of the algebraic parametric equations that describe the asymptotic paths. Since 
we want those relations to be simple, we will avoid fractional powers (which are typical in computations that involve Puiseaux expansions).
{\bf An $F$-asymptotic value $(a,b)\in\mathbb{C}^2$} is a limiting value of $F:\,\mathbb{C}^2\rightarrow\mathbb{C}^2$ along a smooth path
that tends to infinity. The set of all the $F$-asymptotic values is called {\bf the $F$-asymptotic variety}. We denote this set by $A(F)$.
(Notice that the exceptional value set $A_F$ of $F$ is different from $A(F)$. This later set was used by van den Dries. See remark 
\ref{rem1}.) Any smooth curve as above is called {\bf an asymptotic path of $F$ that corresponds to $(a,b)$}.

\begin{remark}\label{rem3} 
Sometimes the definition of asymptotic value uses instead of smooth paths that tend to infinity sequences $(x_n,y_n)\in\mathbb{C}^2$ that
tend to infinity, such that their $F$-image, $\{F(x_n,y_n)\}_{n=1}^{\infty}$ tend to $(a,b)$. Do we get the same notion of asymptotic values
when we use asymptotic paths versus when using sequences $\{(x_n,y_n)\}_{n=1}^{\infty}$? \underline{\bf NO!} Let $F(z)=e^z$. Then the only
asymptotic value of $F(z)$ using smooth paths is $0=0+i0$. Any path $x(t)+iy(t)$, $t>0$, where $x(t)^2+y(t)^2\rightarrow\infty$ and where
$x(t)<0\,\,\,\forall\,t>0$ while $|y(t)|$ is a bounded function is an asymptotic path of $F(z)$. On the other hand the total set of
sequential asymptotic values of $F(z)$ is huge. It equals $\mathbb{C}$. For let $z_0\in\mathbb{C}^{\times}$. Then $z_0=R_0e^{i\theta_0}$
for some $R_>0$. We take the asymptotic sequence $a_n=\left\{\log\,R_0+i(\theta_0+2\pi\cdot i\cdot n\right\}_{n=1}^{\infty}$. For the 
sequential asymptotic value $z_0$ we take $a_n=-n$, say.
\end{remark}
\noindent
However, we have for polynomial mappings the following,

\begin{theorem}\label{sequential-path-asymptotics}
If $F(X,Y)=(P(X,Y),Q(X,Y))\in\mathbb{C}[X,Y]^2$, then the set of all the $F$-asymptotic values along smooth asymptotic paths $A(F)$ is 
identical with the set of all the $F$-sequential asymptotic values (denoted by) $A_S(F)$.
\end{theorem}
\noindent
{\bf Proof.} \\
We use the theory of maximal domains of injectivity of the polynomial mappings $F:\,\mathbb{C}^2\rightarrow\mathbb{C}^2$. For \'etale
mappings the $\mathbb{C}^2$-tilling by $F$-maximal domains is a finite tiling, i.e. the number of tiles is finite. This implies that
the one dimensional theory works and it implies that $A(F)=A_S(F)$.      $ \qed $ \\
\\
We now bring a basic theorem the gives the structure of the asymptotic variety $A(F)$, of a polynomial mapping 
$F:\,\mathbb{C}^2\rightarrow\mathbb{C}^2$. For the readers convenience we included a full detailed proof as appeared in
\cite{p2005}.

\begin{theorem}\label{A(F)}
If $F$ is a two dimensional polynomial mapping, and if $F$ satisfies {\bf the $Y$-degree condition}:
$\deg F=\deg_Y P=\deg_Y Q$, $(F=(P,Q)$) then $\forall\,(a,b)\in A(F)$ there exists a rational
mapping: $R(X,Y)=(X^{-\alpha},X^{\beta}Y+X^{-\alpha}\phi(X))$, ($\alpha\in\mathbb{Z}^+,\,\beta\in
\mathbb{Z}^+\cup\{0\},\,\phi(X)\in\mathbb{C}[X],\,\deg\phi<\alpha+\beta$ and the gcd of the powers of $X$ that
effectively appear in $X^{\alpha+\beta}Y+\phi(X)$ is $1$) with the following two properties: \\
(i) $\exists\,G_R\in\mathbb{C}[X,Y]^2$ such that $G_R=F\circ R$ off ${\rm sing}(R)=\{X=0\}$. (The
mapping $G_R$ is called {\bf the $R$-dual of $F$}) \\
(ii) $\exists\,Y_0\in\mathbb{C}$ such that $(a,b)=G_R(0,Y_0)$. \\
Moreover, the set of all those rational mappings $R$ can be chosen to be a finite set.
\end{theorem}
\noindent
{\bf Proof.} \\
We consider the transformed  mapping:
$$
F_1(U,V)=F\left(\frac{1}{U},\frac{V}{U}\right)=\frac{A(U,V)}{U^{N}},
$$
where $N=\deg F$, $A(U,V)\in\mathbb{C}[U,V]^2$, $\deg A(U,V)=\deg A(0,V)=N$. Any asymptotic value of $F$ 
is a limiting value of $F_1(U,V)$ as $U\rightarrow 0$ and $V$ remains bounded. 
We use the following representation:
$$
F_1(U,V)=\frac{\sum_{j=0}^{N} A_{j}(V)U^{j}}{U^{N}}.  \eqno(*)
$$
In this notation $A_j(V)\in\mathbb{C}[V]^2$, $0\le j\le N$. These are the two polynomial coefficients of $U^j$.
When $U\rightarrow 0$ and $V$ remains bounded we will have $F_1(U,V)\rightarrow
\infty$ unless $V$ tends to a zero of the free term $A_{0}(V)$. This means a simultaneous zero of the 
pair of polynomials $A_0(V)$. The reason is that if $V$ is bounded away from zeros of $A_{0}(V)$, 
then the term $A_{0}(V)/U^{N}$ dominates the finite sum $F_1(U,V)$ (in equation (*)) when $U\rightarrow 0$.
We conclude that in order to determine all the possible
asymptotic values of $F$, we should consider the limits of $F_1(U,V)$ when
$U\rightarrow 0$ and $V\rightarrow a_0$, where $a_0$ is a zero of $A_{0}$. 
Let us represent the coefficients of $A$ as follows, $ A_{j}(V)=(V-a_0)^{p_{j}}B_{j}(V),\,\,\,\,\,\,0\le j\le N$.
If $A_{j}\equiv (0,0)$ we agree that $p_{j}=\infty $ and otherwise $p_{j}\ge 0$
and $B_{j}(a_0)\ne (0,0)$. Next we make the transformation, $V=a_0+WU^{p}$,
where $p$ will be a positive number that will be determined soon (in fact it is going
to be rational). Putting together everything we obtain,
$$
F_1(U,V)=\frac{\sum_{j=0}^{N} B_{j}(V)W^{p_{j}}U^{pp_{j}+j}}{U^{N}}.
$$
We want to determine the desired value of $p$ that will lead to a finite limiting value
of $F_1(U,V)$. If $p$ is a small positive number, then $\forall\,1\le j\le N$, $pp_{0}<pp_{j}+j$.
Thus the term that will dominate the finite sum that represents $F_1(U,V)$ is $B_{0}(a_0)W^{p_{0}}/
U^{N-pp_{0}}$. If $pp_{0}=N$ and $pp_{0}<pp_{j}+j$ for $1\le j\le N$, then
with this choice of $p=N/p_{0}$ we obtain the family of limiting values
$B_{0}(a_0)W^{p_{0}}$. In the complementary case, we choose $p>0$ so 
that $pp_{0}\le pp_{j}+j$ for $1\le j\le N$ with equality for at least
one value $j_{0}$ of $j$. We can write this choice of a value for $p$ as follows,
$$
p=\min\left\{\frac{j}{p_{0}-p_{j}}\,|\,1\le j\le N,\,p_{0}>p_{j}\right\}.
$$
We substitute this and obtain:
$$
F_1(U,a_0+WU^{p})=\frac{C(U,W)}{U^{N-pp_{0}}},
$$
where $C(U,W)$ is a mapping with coordinates that are polynomial in $U$ and $U^{p}$. It is also a polynomial in
$W$ of degree $N$ or less. Finally $C(0,W)$ is a polynomial in $W$ of degree 
$p_{0}$, which contains only the powers $W^{p_{j}}$ that satisfy the condition
$pp_{0}=pp_{j}+j$. Let $p=b_0/c$ where $b_0$ and $c$ are coprime positive integers.
Then to avoid fractional powers we make further the substitution $U=Z^{c}$.
We get,
$$
F_1(Z^{c},a_0+WZ^{b_0})=\frac{D(Z,W)}{Z^{cN-b_0p_{0}}}.   \eqno(**)
$$
Here $D(Z,W)$ is a polynomial pair in $Z$ and $W$ and $D(0,W)=C(0,W)$ is a polynomial
pair in $W$ of degree $p_{0}$. We noted above that the powers $W^{p_{j}}$ that
appear in $D(0,W)$ satisfy $pp_{0}=pp_{j}+j$. Hence for these $j$'s
$$
p_{0}-p_{j}=\frac{c\cdot j}{b_0}.
$$
Since $b_0$ and $c$ are coprime and since $p_{0}-p_{j}\in\mathbb{Z}$, $b_0$ must be a
divisor of $j$, so that the difference $p_{0}-p_{j}$ is a multiple of $c$. We
conclude that $D(0,W)$ is of the form $D(0,W)=W^{\alpha}E(W^{c})$, where $\alpha\in\mathbb{Z}^{+}\cup\{ 0\}$, 
and $E(X)$ is a polynomial pair in $X$ of degree 1 or more. \\
We see that the form of $F_1$ after our transformations, as given in equation
(**) is of the same type in the indeterminates $Z$ and $W$ as it was in equation
(*) in the indeterminates $U$ and $V$. Thus we can repeat the sequence of
transformations with suitable parameters. We conclude that any limiting value
of $F_1$ will be attained along a curve of the form $V=a_0+WU^{p}$, where 
$U\rightarrow 0$ and $W$ being bounded and $p$ is a positive rational number.
If the parameter $p$ is chosen to be smaller than our minimum value formula
then $F_1\rightarrow\infty $ along the curve. So finite limiting values for
$F_1$ can be achieved only for this value or larger ones for $p$. So with that
minimum chosen value for $p$, asymptotic values of $F_1$ that correspond to
the zero $a_0$ of the pair $A_{0}$ are achieved when $Z\rightarrow 0$ and $W$ being bounded.
If the value of $p$ is chosen to be larger than the minimum then, in fact
$W\rightarrow 0$. On repetition of the process, each asymptotic value will 
correspond to a zero $a_{1}$ of $D(0,W)$. If the multiplicity of the zero $a_{1}$ 
is $q_{0}$, then the transformation we apply is $W=a_{1}+SZ^{q}$, where $q$
is chosen using our minimum process (as with $p$). We have $q_{0}\le p_{0}
\le N$. The case $q_{0}=p_{0}$ implies $D(0,W)=d\cdot (W-a_{1})^{p_{0}}$ with a non
zero constant $d$. So $D(0,W)$ contains all the powers of $W$ from 0
to $p_{0}$. Hence the denominator of $q$, $c=1$ and the power of $Z$ in
the denominator is $N-p_{0}$ or less. We conclude, that repeating the process
will in every iteration either strictly reduce the degree of the leading
term, or strictly reduce the power of the denominator by a positive integer.
We conclude that the process must terminate. The way it will terminate is 
as follows. If we use the notation $A(U,V)/U^{N}$, the final transformation 
must be of the form $V=a_0+WU^{p}$ with $p=N/p_{0}$. The other possibility is that 
this choice of value for $p$ coincides with the minimum $\min \{j/(p_{0}-p_{j})\}$.
We obtain a curve of asymptotic values $D(0,W)$, a polynomial in $W$ of
degree $p_{0}\ge 1$. For as the denominator in the last iteration becomes 1,
our sequence of transformations assigns a new polynomial $D(Z,W)$ to the
original $F(X,Y)$. Clearly if we perform the process to every zero of the leading
term at every stage we will obtain all the asymptotic values of $F$.
Clearly we obtain in this way finitely many rational mappings of the type prescribed in
the statement of our theorem. We obtain those mappings by composing the sequence of the transformations (and invert
the result). In general asymptotic values of $F$ will be achieved as limits of $F$
along more than one of the rational curves in our construction. $\qed $ \\
\begin{remark}\label{no-need-for-etale}
We did not use the Jacobian condition in our proof of Theorem \ref{A(F)}. The theorem is valid for polynomial mappings that are not
necessarily Keller. The proof shows the following point: the asymptotic paths we constructed in Theorem \ref{A(F)} are algebraic
in the coefficients of $F$. We only used rational transformations and zeros of $A_0(V)$'s. All the powers in our construction are integers,
i.e. no non-integral fractional powers are in our formulas. Soon we will be interested in the number of points of the lattice
$\mathbb{Z}^2$ on those asymptotic paths. The above implies that there will not be apriori cases in which this diophantine problem
is trivial.
\end{remark}

\begin{theorem}\label{R_0(F)}
Let $F\in\mathbb{C}[X,Y]^2$ be a Keller mapping which is not a $\mathbb{C}^2$-automorphism. Then there exists a non-empty set $R_0(F)$ of
rational non-polynomial mappings $R\in\mathbb{C}(X,Y)^2-\mathbb{C}[X,Y]^2$ that satisfies the following: \\
{\rm (i)} The cardinality of $R_0(F)$ equals the number of the irreducible components of the $F$-asymptotic variety $A(F)$. \\
{\rm (ii)} $\forall\,R\in R_0(F)$, the mapping $G_R=F\circ R$ has a polynomial extension to $\mathbb{C}^2$. \\
{\rm (iii)} $\forall\,R\in R_0(F)$, the set $\{G_R(a,b)|\,(a,b)\in {\rm sing}(R)\}$ is an irreducible component of $A(F)$. \\
{\rm (iv)} We have the representation $A(F)=\bigcup_{R\in R_0(F)}\{G_R(a,b)|\,(a,b)\in {\rm sing}(R)\}$.
\end{theorem}
\noindent
{\bf Proof.} \\
We can choose a linear invertible mapping $L(U,V)$ so that $F\circ L$ satisfies the $V$-degree condition. The construction that is outlined
in the proof of Theorem \ref{A(F)} gives us finitely many rational non-polynomial mappings (in the current case at least one) of the form
$R(X,Y)=L(X^{-\alpha},X^{\beta}Y+X^{-\alpha}\phi(X))$. This aggregate of mappings satisfy $G_R=F\circ R\in\mathbb{C}[X,Y]^2$ off 
${\rm sing}(R)=\{X=0\}$. This proves {\rm (ii)}. Any asymptotic value of $F$ equals $G_R(0,Y)$ for some $R$ and $Y\in\mathbb{C}$, thus
proving {\rm (iv)}. The polynomial parametrization $\{G_R(0,Y)|\,Y\in\mathbb{C}\}$ is a normal parametrization (i.e. a surjective one) of
one of the irreducible components of $A(F)$ and we conclude the proof of Theorem \ref{R_0(F)}.     $  \qed  $ \\
\begin{remark}\label{non-empty-R_0(F)}
$R_0(F)\ne\emptyset$ because by the assumption on $F$ in Theorem \ref{R_0(F)}, $F$ is a polynomial non-proper mapping because it is an
\'etale but not a $\mathbb{C}^2$-automorphism. Furthermore, each component of $A(F)$ is a polynomial plane curve (i.e. it has exactly
one place at infinity). Finally we note that $R(X,Y)=L(X^{-\alpha},X^{\beta}Y+X^{-\alpha}\phi(X))$ is non-singular (in fact holomorphic)
in its domain of definition for $(\partial X^{-\alpha})/(\partial X)$ never vanishes for $X\ne 0$.
\end{remark}

\begin{definition}\label{geometric-basis}
If $F\in\mathbb{C}[X,Y]^2$ satisfies the assumptions of Theorem \ref{R_0(F)}, then $R_0(F)$ is called {\bf a geometric basis of $F$}.
\end{definition}
\noindent
We now come to those theorems that justify the title of Section \ref{sec2}. We need to recall the following result:

\begin{theorem}\label{fiber-theorem} {\rm(Fiber Theorem for Polynomial Maps, \cite{a1994}, page 32)} \\
Let $K$ be a commutative integral domain of characteristic zero and suppose that $F:=(F_1,F_2,\ldots,F_n)$ is a locally invertible system 
of polynomials in $K[X]$, $X:=(X_1,X_2,\ldots,X_n)$. Then the geometric degree $\dim_{K(F)}K(X)$ is finite. Moreover, for any point
$b\in K^n$, $\dim_{K(F)}K(X)$ majorizes the cardinality of the fiber $F^{-1}(b)$.
\end{theorem}
\noindent
To explain the notion of {\bf "$F$ is a locally invertible system of polynomials in $K[X]$"}, we can look at:

\begin{proposition}\label{locally-invertible-system} {\rm (\cite{a1994}, page 29)} \\
For $F\in K[X]^n$, the following conditions are equivalent: \\
{\rm (i)} $F$ is a locally invertible system of polynomials in $K[X]$. \\
{\rm (ii)} $\det\,J(F)(a)\in K^{\times}$ for every $a\in K^n$.
\end{proposition}
\noindent
We may also look at the:

\begin{definition}\label{invertible and locally-invertible systems} {\rm (\cite{a1994}, page 29)} \\
An element $F=(F_1,\ldots,F_n)$ of $K[[X]]^n$ is called an invertible system of formal series of $K[[X]]$or more briefly an invertible
element of $K[[X]]^n$ if $F(0)=0$ and there exits $G$ in $K[[X]]^n$ called the formal inverse of $F$, such that $G(0)=0$, $F(G)=G(F)=X$.
Now, let $F=(F_1,\ldots,F_n)\in K[X]^n$. For $a\in K^n$ we call the reduced Taylor polynomial of $F$ at the point $a$ the polynomial
$F_a=F_a(X)=F(X+a)-F(a)$. We say that $F$ is locally invertible system of polynomials of $K[X]$ or more briefly a locally invertible 
element of $K[X]^n$ if the reduced Taylor polynomial of $F$ at each point of $K^n$ is an invertible system of formal series of $K[[X]]$.
We say that $F$ is an invertible system of $K[X]$ or more briefly an invertible element of $K[X]^n$ if there exists $G\in K[X]^n$, called
"the inverse of $F$", such that $F(G)=G(F)=X$.
\end{definition}
\noindent
Often the fiber theorem for polynomial mappings is expressed as follows: \\
For dominant polynomial mapping over a field of characteristic zero, the cardinality of any fiber (pre-image of a point) is bounded
above by the degree of the field extension determined by the mapping. \\
We note that the uniform finiteness of the cardinality of fibers of a two dimensional \'etale mapping can be deduced from the B\'ezout
Theorem: \\
If $F=(P,Q)\in\mathbb{C}[X,Y]^2$ satisfies $JF\in\mathbb{C}^{\times}$, then $\forall\,(a,b)\in\mathbb{C}^2$, we have 
$|F^{-1}((a,b))|\le (\deg\,P)(\deg\,Q)$. \\
{\bf Proof.} \\
We consider the number of solutions of the polynomial system:
$$
\left\{\begin{array}{c} P(X,Y)=a \\ Q(X,Y)=b \end{array}\right.
$$
This is the number of intersection points of the two planar algebraic curves. The Jacobian condition $J(F)\in\mathbb{C}^{\times}$ and
the B\'ezout Theorem imply the inequality of the Fiber-Theorem in this special case.  $   \qed  $ \\
\\
Let $F=(P_1,\ldots,P_n)\in(\mathbb{Z}[X_1,\ldots,X_n])^n$ be \'etale (here we insist on $J(F)\equiv 1$). Then $F$ is dominant. 
For $n=2$ we have in fact $|\mathbb{C}^2-F(\mathbb{C}^2|\le 2$. If $F$ is not injective then by a famous theorem of J. Hadamard (could
be found in his 1906 Selecta), $F$ must have asymptotic values. We will prove that not even one of the infinitely many asymptotic 
values of $F$ can be realized along a sequence of points in the lattice $\mathbb{Z}^n$. We can say more:

\begin{theorem}\label{no-asymptotic-values-along-lattice-points}
If $F\in(\mathbb{Z}[X_1,\ldots,X_n])^n$ is \'etale and not injective and if $A_1,A_2,A_3,\ldots$ is an infinite sequence of different
points in the lattice $\mathbb{Z}^n$, then $\lim_{m\rightarrow\infty}||F(A_m)||_2=\infty$.
\end{theorem}
\noindent
{\bf Proof.} \\
If not, then there is a sub-sequence $B_1,B_2,B_3,\ldots$ of $A_1,A_2,A_3,\ldots$ such that $\lim_{m\rightarrow\infty}F(B_m)=L$ exists
and is finite. Since $F\in(\mathbb{Z}[X_1,\ldots,X_n])^n$ and $\forall\,m,\,B_m\in\mathbb{Z}^n$, it follows that 
$\forall\,m,\,F(B_m)\in\mathbb{Z}^n$ and so there is an index $N$ such that $\forall\,m\ge N,\,F(B_m)=L$. Hence the fiber $F^{-1}(L)$
is infinite. This contradicts the Fiber Theorem for polynomial mappings.  $  \qed  $ \\
\\
Theorem \ref{no-asymptotic-values-along-lattice-points} can be generalized as follows:

\begin{theorem}\label{generalization}
Let $F\in(\mathbb{Z}[X_1,\ldots,X_n])^n$ be \'etale and not injective. If $A_1,A_2,A_3,\ldots$ is an infinite sequence of different points
in the  lattice $(\mathbb{Z}/N)^n$ where $N\in\mathbb{Z}_{>0}$ is fixed. Then $\lim_{m\rightarrow\infty}||F(A_m)||_2=\infty$.
\end{theorem}
\noindent
{\bf Proof.} \\
If not, then we can find a sub-sequence $B_1,B_2,B_3,\ldots$ of the given sequence $A_1,A_2,A_3,\ldots$ such that 
$\lim_{m\rightarrow\infty}F(B_m)=L$ exists. Let $d=\deg\,F$. Then $\forall\,m,\,\,F(B_m)\in (\mathbb{Z}/N^d)^n$ and so there is an index $M$
such that $\forall\,m\ge M,\,\,F(B_m)=L$. This contradicts the Fiber Theorem for polynomial mappings.   $   \qed   $ \\
\\
\begin{remark}\label{alternative-proof}
We can give an alternative proof for Theorem \ref{generalization} by reducing it to the case of Theorem 
\ref{no-asymptotic-values-along-lattice-points}. We define,
$$
G(X_1,\ldots,X_n)=N^d\cdot F\left(\left(\frac{X_1}{N},\ldots,\frac{X_n}{X_N}\right)\right)\,\,\,\,\,\,\,{\rm where}\,\,d=\deg\,F.
$$
Then $G\in(\mathbb{Z}[X_1,\ldots,X_n])^n$ and for any sequence $A_1,A_2,A_3,\ldots$ in $(\mathbb{Z}/N)^n$ we have,
$$
\lim_{m\rightarrow\infty}||F(A_m)||_2=N^{-d}\lim_{m\rightarrow\infty}||G(N\cdot A_m)||_2,
$$
where $NA_1,NA_2,NA_3,\ldots$ is a sequence in $\mathbb{Z}^n$. Thus Theorem \ref{no-asymptotic-values-along-lattice-points}
applies to the data $G$ and $NA_1,NA_2,NA_3,\ldots$ and we get $\lim_{m\rightarrow\infty}||G(A_m)||_2=\infty$.
\end{remark}

\begin{remark}\label{interpretation}
Theorem \ref{no-asymptotic-values-along-lattice-points} and Theorem \ref{generalization} say that if $F\in(\mathbb{Z}[X_1,\ldots,X_n])^n$
is \'etale and not injective then any asymptotic path $\{(\alpha_1(t),\ldots,\alpha_n(t))|\,0<t\}$ of $F(X_1,\ldots,X_n)$, i.e.
$F(\alpha_1(t),\ldots,\alpha_n(t))\rightarrow_{t\rightarrow\infty}(\beta_1,\ldots,\beta_n)$, 
where $|\alpha_1(t)|^2+\ldots+|\alpha_n(t)|^2\rightarrow_{t\rightarrow\infty}\infty$, contains only finitely many lattice points
$\mathbb{Z}^n$. This means that each such an asymptotic path misses all the points of $\mathbb{Z}^n$ except for finitely many of them.
We recall that by Theorem \ref{A(F)} all the coordinates $\alpha_1(t),\ldots,\alpha_n(t)$ are algebraic in the coefficients of 
$F(X_1,\ldots,X_n)$. So the system $\alpha_j(t)=m_j$, $1\le j\le n$ of algebraic equations in the unknown $t\in\mathbb{C}$ have only finitely many integral solutions, i.e. as diophantine system. Only finitely many $n$-tuples $(m_1,\ldots,m_n)\in\mathbb{Z}^n$ give
solvable diophantine equations.
\end{remark}

\section{Theta functions and iterations of polynomial mappings $\mathbb{R}^2\rightarrow\mathbb{R}^2$ which have positive determinant
of their Jacobian matrix.}\label{sec3}

This section and the next two form the last part of our paper. We describe in this part one more arithmetical application to polynomial
mappings in dimension two. This time we analyze the dynamical system constructed by iterations of certain polynomial mappings (related
to the two dimensional Jacobian Conjecture, over $\mathbb{R}$ and also over $\mathbb{C})$). The arithmetical tool used is theory of
Theta functions. The flavor here in of analytic number theory.

\begin{definition}\label{k-transformations} (\cite{p1996}, page 107 and \cite{p1998}, page 10) \\
Let $f:\,\mathbb{R}^2\rightarrow\mathbb{R}^2,\,\,\,\,\,f(X,Y)=(P(X,Y),Q(X,Y))$ be a real polynomial map. Let $k$ be a non-negative integer
such that
\begin{equation}\label{eq-(3)}
\deg_X P< 2k+1.
\end{equation}
The standard $k$-transformation (a change of variables) is defined by
\begin{equation}\label{eq-(4)}
T_k:\,(X,Y)\rightarrow (U,(1+U^{2k})V).
\end{equation}
The $k$-normalized form of $f$ is the following polynomial map
$$
\left\{\begin{array}{l} F_k=f\circ T_k:\,\mathbb{R}^2\rightarrow\mathbb{R}^2 \\ F_k(U,V)=f(U,(1+U^{2k})V). \end{array}\right.
$$
\end{definition}
\noindent
We have the following:

\begin{theorem}\label{X,Y-finite} {\rm (\cite{p1996}, page 107 and \cite{p1998}, page 11)}\\
Let $f:\,\mathbb{R}^2\rightarrow\mathbb{R}^2$ be a polynomial map that satisfies the Jacobian condition: $\forall\,(X,Y)\in\mathbb{R}^2$,
$\det\,J(f)(X,Y)\ne 0$. Then also the $k$-normalized form of $f$, $F_k$, satisfies this Jacobian condition and $f$ is a global diffeomorphism
if and only if $F_k$ has no $X$-finite or $Y$-finite asymptotic values.
\end{theorem}

\begin{definition}\label{X,Y finite and weak,strong etale}
An asymptotic path $\sigma(t)=(X(t),Y(t))$, $0\le t<\infty$, of a polynomial $P(X,Y)\in\mathbb{R}[X,Y]$ or of a polynomial map $(P,Q)$
is called $X$-finite if $\lim_{t\rightarrow\infty} |Y(t)|=\infty$ while $-\infty<\lim_{t\rightarrow\infty} X(t)<\infty$. The corresponding
asymptotic value, $\lim_{t\rightarrow\infty} P(X(t),Y(t))$ or  $\lim_{t\rightarrow\infty} f(X(t),Y(t))$ ($f(X,Y)=(P(X,Y),Q(X,Y))$)
is called an $X$-finite asymptotic value. \\
$f:\,\mathbb{R}^2\rightarrow\mathbb{R}^2$ is weakly \'etale if $J(F)(X,Y)>0$ and it is strongly \'etale if there is a constant $\epsilon>0$
such that $J(F)(X,Y)\ge\epsilon$.
\end{definition}
\noindent
The family $\{T_k\}_{k\in\mathbb{Z}_>0}$ with composition of mappings as the binary operation is commutative. This family is not closed
to composition of mappings: 
$$
(T_n\circ T_m)(U,V)=(U,(1+U^{2n})(1+U^{2m})V)=
$$
$$
=(U,(1+U^{2m})(1+U^{2n})V)=(T_m\circ T_n)(U,V).
$$
We are interested in the finite dynamical systems: $(f\circ T_m\circ\ldots\circ T_n)(U,V)$ and in their limit when $n\rightarrow\infty$.

\begin{proposition}\label{dynamical-formula}
Let $k\le m\le n$, then $(f\circ T_m\circ\ldots\circ T_n)(U,V)=f(U,(\prod_{j=m}^n(1+U^{2j}))\cdot V)$.
\end{proposition}
\noindent
If $U\in\mathbb{R}$ then $\prod_{j=m}^n(1+U^{2j})\ge 1$, with equality if and only if $U=0$. But if we allow $U\in\mathbb{C}$ then
$1+U^{2j}=0$ if and only if 
$$
U=U_N:=\exp\left(\left(\frac{1+2N}{2j}\right)\pi\cdot i\right),\,\,\,\,{\rm for}\,\,{\rm any}\,\,N\in\mathbb{Z}.
$$
if $j\in\mathbb{Z}^{\times}$ then $U_{N_1}=U_{N2}$ if and only if $\exists\,M\in\mathbb{Z}$ such that $N_2-N_1=2Mj$. The smallest 
non-negative $M$ is $M=0$. So there are exactly $j$ different $U_N$'s. These are $U_0,\,U_1,\ldots,U_{j-1}$ and they give us the
following full family of $U_N$'s that differ from one another:
$$
\exp\left(\frac{1}{2j}\cdot\pi i\right),\,\exp\left(\frac{3}{2j}\cdot\pi i\right),\,\exp\left(\frac{5}{2j}\cdot\pi i\right),\ldots,
\exp\left(\frac{2j-1}{2j}\cdot\pi i\right).
$$
The parameter $j$ is subject only to the restriction $k\le j$. If we agree to write only the factors 
$$
\frac{1}{2j},\,\frac{3}{2j},\,\frac{5}{2j},\ldots,\frac{2j-1}{2j}
$$
within the exponential, we get the following lower triangular infinite matrix:
$$
\begin{array}{lllllll} j=1 & 1/2 & & & & &  \\ j=2 & 1/4 & 3/4 & & & &  \\ j=3 & 1/6 & 3/6 & 5/6 & & &  \\
j=4 & 1/8 & 3/8 & 5/8 & 7/8 & &  \\ j=5 & 1/10 & 3/10 & 5/10 & 7/10 & 9/10 &  \\
j=6 & 1/12 & 3/12 & 5/12 & 7/12 & 9/12 & 11/12 \end{array}
$$
There are overlaps among sequences of different $j$ values. The total set of zeros for all the $j$'s is:
$$
\{U\in\mathbb{C}|\,\exists\,j\ge k\,\,{\rm such}\,\,{\rm that}\,\,1+U^{2j}=0\}=
$$
$$
=\left\{\exp\left(\left(\frac{1+2N}{2j}\right)\pi\cdot i\right):\,N\in\mathbb{Z},\,j\in\mathbb{Z}_{>0},\,j\ge k\right\}.
$$
This is a countable dense subset of the unit circle $\mathbb{T}=\{Z\in\mathbb{C}|\,|Z|=\}$ (in the strong topology). We consider
the complex function which is defined on the square of the unit  disc $\mathbb{D}^2=\{Z\in\mathbb{C}|\,|Z|<1\}$, by the following
infinite composition of holomorphic functions in $\mathbb{D}^2$ by:
$$
(f\circ T_k\circ T_{k+1}\circ T_{k+2}\circ\ldots)(U,V)=f\left(U,\left(\prod_{j=k}^{\infty}(1+U^{2j})\right)V\right)=
$$
$$
=\lim_{n\rightarrow\infty}f\left(U,\left(\prod_{j=k}^{n}(1+U^{2j})\right)V\right)=
$$
$$
=f\left(U,\lim_{n\rightarrow\infty}\left(\prod_{j=k}^{\infty}(1+U^{2j})\right)V\right).
$$
We used in the last step the convergence of the infinite product uniformly on compacta of $\mathbb{D}$. The fact that
the infinite product is univariate (just $U$) and appears once in the second coordinate of $f$ helps. We remember that 
$f(X,Y)\in\mathbb{C}[X,Y]^2$ is a polynomial map. Also over $\mathbb{R}$ $T_j((U,V)=(U,(1+U^{2j})V)$ so $J(T_j)=1+U^{2j}\ge 1$
and by that and by the chain rule, $f$ is weakly \'etale if and only if $f\circ T_j$ is weakly \'etale, and similarly
$f$ is strongly \'etale if and only if $f\circ T_j$ is strongly \'etale. We know that the limit of a sequence of analytic functions
of a single complex variable which converge uniformly on compact subsets of a domain exists. This limit is analytic function. It might be 
degenerate (constant). In our case the infinite product $\prod_{j=k}^{\infty}(1+U^{2j})$ can not be a constant. Moreover, the polynomial
mapping $f(U,V)$ is \'etale and so depends on both coordinates $(U,V)$ and hence the limit of our dynamical system is a non-degenerate 
mapping in $(U,V)$. We will see how to expand 
$$
f\left(U,\left(\prod_{j=k}^{\infty}(1+U^{2j})\right)V\right)
$$
in terms of the Theta functions $\Theta_1$ and $\Theta_4$ (known to experts on Theta functions) and derive interesting identities that involve $f$. We first would like to point at a very strong analytic property of the limit of our dynamical system. Namely because of the density of the zeros of the infinite product (in $U$) on $\mathbb{T}$ we deduce that the unit circle $\mathbb{T}$ is a natural boundary 
of $\prod_{j=k}^{\infty}(1+U^{2j})$, $|U|<1$. Hence the sequence of the \'etale polynomial maps 
$(f\circ T_k\circ T_{k+1}\circ T_{k+2}\circ\ldots)(U,V)$ can not be analytically extended beyond the boundary portion $\mathbb{T}$, in $U$.

\begin{proposition}\label{no X,Y finite asymptotic values}
Let $f:\,\mathbb{R}^2\rightarrow\mathbb{R}$be a polynomial mapping. Let $k\in\mathbb{Z}_{\ge 0}$ satisfy equation (\ref{eq-(3)}). Then
for any $n,\,m$ such that $k\le m\le n$, $f(U,V)$ is a global diffeomorphism on $\mathbb{R}^2$ if and only if 
$f(U,(\prod_{j=m}^{n}(1+U^{2j}))\cdot V)$ has no $X$-finite or $Y$-finite asymptotic values.
\end{proposition}
\noindent
{\bf Proof.} \\
We use finite induction on Theorem \ref{X,Y-finite}.  $  \qed  $

\section{Infinite products of certain arithmetical functions.}\label{sec4}

This section reviews some results (mostly computational) in the theory of Theta functions. More concretely we will look at highly
non-trivial formulas that give the power series expansion of a given infinite product of functions. Some of these identities are named
after Euler and Jacobi.
$$
\prod_{k=1}^{\infty}(1+q^k)=\sum_{k=0}^\infty P_d(k)q^k,\,\,\,\,\,\,\,|q|<1.
$$
Here $P_d(k)$ is the number of partitions of $k$ with distinct parts. $P_d(0)=1$. This proves the following:

\begin{lemma}\label{lemma1}
$$
\prod_{j=k}^{\infty}(1+U^{2j})=\prod_{j=1}^{\infty}(1+U^{2j})/\prod_{j=1}^{k-1}(1+U^{2j})=
\sum_{j=0}^{\infty}P_d(j)U^{2j}/\prod_{j=1}^{k-1}(1+U^{2j}),
$$
$|U|<1,\,\,\,k\ge 2$.
\end{lemma}
\noindent
Next we mention the following three identities:
\begin{equation}\label{eq1}
\prod_{j=1}^{\infty}(1-q^j)=\sum_{n=-\infty}^{\infty}(-1)^n q^{\frac{n(3n-1)}{2}},\,\,\,\,\,\,\,|q|<1.
\end{equation}

\begin{equation}\label{eq2}
\prod_{j=1}^{\infty}(1-q^j)^3=\sum_{n=0}^{\infty}(-1)^n(2n+1)q^{\frac{n(n+1)}{2}},\,\,\,\,\,\,\,|q|<1.
\end{equation}
\noindent
That is a result of Jacobi.

\begin{equation}\label{eq3}
\prod_{j=1}^{\infty}\left(\frac{1-q^j}{1+q^j}\right)=1+2\sum_{n=1}^{\infty}(-1)^n q^{n^2},\,\,\,\,\,\,\,|q|<1.
\end{equation}
\noindent
By equations (\ref{eq1}) and (\ref{eq3}):

\begin{equation}\label{eq5}
\prod_{j=1}^{\infty}(1+q^j)=\sum_{n=-\infty}^{\infty}(-1)^n q^{\frac{n(3n-1)}{2}}/\left\{1+2\sum_{n=1}^{\infty}(-1)^n q^{n^2}\right\},\,\,\,
|q|<1.
\end{equation}
\noindent
By equations (\ref{eq2}) and (\ref{eq3}):

$$
\prod_{j=1}^{\infty}(1+q^j)=\left\{\sum_{n=0}^{\infty}(-1)^n(2n+1)q^{\frac{n(n+1)}{2}}\right\}^{1/3}/
\left\{1+2\sum_{n=1}^{\infty}(-1)^n q^{n^2}\right\},
$$
\begin{equation}\label{eq6}
|q|<1.
\end{equation}
\noindent
Using equation (\ref{eq5}) we obtain:

\begin{lemma}\label{lemma2}
$$
\prod_{j=k}^{\infty}(1+U^{2j})=\sum_{n=-\infty}^{\infty}(-1)^n U^{n(3n-1)}/\left\{\left(1+2\sum_{n=1}^{\infty}(-1)^n U^{2n^2}\right)
\prod_{j=1}^{k-1}(1+U^{2j})\right\},
$$
$|U|<1,\,\,\,k\ge 2$.
\end{lemma}
\noindent
Using equation (\ref{eq6}) we obtain:

\begin{lemma}\label{lemma3}
$$
\prod_{j=k}^{\infty}(1+U^{2j})=\left\{\sum_{n=0}^{\infty}(-1)^n(2n+1)U^{n(n+1)}\right\}^{1/3}/
\left\{\left(1+2\sum_{n=1}^{\infty}(-1)^n U^{2n^2}\right)\prod_{j=1}^{k-1}(1+U^{2j})\right\},
$$
$|U|<1,\,\,\,k\ge 2$.
\end{lemma}

\section{A theorem on Theta functions and \'etale mappings with coordinates in $\mathbb{Z}[X,Y]$.}\label{sec5}

Using arguments that led us to Lemma \ref{lemma1}, Lemma \ref{lemma2} and Lemma \ref{lemma3} we arrive at theorems of the following type:

\begin{theorem}\label{thm2} {\rm (\'etale mappings and Theta functions)}
Let $f:\,\mathbb{R}^2\rightarrow\mathbb{R}^2$, $f(X,Y)=(P(X,Y),Q(X,Y))\in\mathbb{R}[X,Y]^2$. Let $k\in\mathbb{Z}_{>0}$ be a natural number
that satisfies $\deg_X P<2k+1$. We define the $k$-transformation by $T_k:\,\mathbb{R}^2\rightarrow\mathbb{R}^2$, $T_k(U,V)=(U,(1+U^{2k})V)$.
We define the $k$-normalized form of $f$ by $F_k=f\circ T_k$. {\rm See the beginning of section \ref{sec3}}. Then: \\ \\
{\rm (1)} $f$ satisfies the Jacobian condition in the weak form: $J(f)(X,Y)\ne 0$, if and only if $F_k$ satisfies that condition, i.e.
$J(F_k)(X,Y)\ne 0$.
\\ \\
{\rm (2)} $f$ satisfies the Jacobian condition in the strong form: $\exists\,\epsilon>0$ such that $J(f)(X,Y)\ge\epsilon$, if and only if
$F_k$ satisfies that condition with the same $\epsilon>0$, i.e. $J(F_k)(X,Y)\ge\epsilon$.
\\ \\
{\rm (3)} $f$ is an injection $\mathbb{R}^2\rightarrow f(\mathbb{R}^2)$, if and only if $F_k$ is an injection $\mathbb{R}^2\rightarrow
F_k(\mathbb{R}^2)=f(\mathbb{R}^2)$.
\\ \\
{\rm (4)} The complex Euler type infinite product $\prod_{j=k}^{\infty}(1+U^{2j})$, $U\in\mathbb{C}$, converges uniformly on compact
subsets of of the unit disc $\mathbb{D}=\{U\in\mathbb{C}|\,|U|<1\}$, and hence defines an analytic function of $U$ in $\mathbb{D}$.
\\ \\
{\rm (5)} The set of all the zeros of each of the factors $1+U^{2j}$, $j\ge k$, is the countable and dense (in the strong topology) subset
of the unit circle $\mathbb{T}=\{U\in\mathbb{C}|\,|U|=1\}$ given by:
$$
\left\{\exp\left(\left(\frac{1+2N}{2j}\right)\pi\cdot i\right):\,N\in\mathbb{Z},\,j\in\mathbb{Z}_{>0},\,j\ge k\right\}.
$$
{\rm (6)} The unit circle $\mathbb{T}$ is a natural boundary for $\prod_{j=k}^{\infty}(1+U^{2j})$, $U\in\mathbb{D}$.
\\ \\
{\rm (7)} The following sequence of complex polynomial maps:
$$
\left\{(f\circ T_k\circ T_{k+1}\circ\ldots\circ T_n)\right\}_{n=k}^{\infty}=
\left\{f\left(U,\left(\prod_{j=k}^{n}(1+U^{2j})\right)\cdot V\right)\right\}_{n=k}^{\infty},
$$
$(U,V)\in\mathbb{D}^2$, converges uniformly on compact subsets of $\mathbb{D}^2$.
\\ \\
{\rm (8)} The set $\mathbb{T}\times\mathbb{D}$ is a natural boundary of the holomorphic mapping,
$$
\lim_{n\rightarrow\infty}f\left(U,\left(\prod_{j=k}^{n}(1+U^{2j})\right)\cdot V\right),\,\,\,\,\,(U,V)\in\mathbb{D}^2.
$$
Analytic extension in the $U$ variable of the infinite product is the obstacle.
\\ \\
{\rm (9)} The holomorphic mapping defined by the limit of the sequence in $\mathbb{D}^2$ (in {\rm (7)} above) has the following three
representations in $\mathbb{D}^2$ below. \\
We denote the univariate polynomial in $U$, of degree $2(k-1)$
$$
S_{2(k-1)}(U)=\prod_{j=1}^{k-1}(1+U^{2j}),\,\,\,\,\,\,\,U\in\mathbb{D}.
$$
Then the first representation is:
$$
f\left(U,V\cdot\sum_{j=0}^{\infty}P_d(j)U^{2j}/S_{2(k-1)}(U)\right).
$$
The second representation is:
$$
f\left(U,V\cdot\sum_{n=-\infty}^{\infty}(-1)^n U^{n(3n-1)}/\left\{\left(1+2\sum_{n=1}^{\infty}(-1)^n U^{2n^2}\right)(S_{2(k-1)}(U))\right\}\right),
$$
and the third representation is:
$$
f\left(U,V\cdot\left\{\sum_{n=0}^{\infty}(-1)^n(2n+1)U^{n(n+1)}\right\}^{1/3}/
\left\{\left(1+2\sum_{n=1}^{\infty}(-1)^n U^{2n^2}\right)(S_{2(k-1)}(U))\right\}\right).
$$
\\ \\
{\rm (10)} The holomorphic mapping in {\rm (9)}, restricted to the square $(U,V)\in (-1,1)^2$ satisfies the same Jacobian condition
(weak or strong) as the polynomial mapping $f(U,V)$. Also, it is an injection $(-1,1)\rightarrow f((-1,1))$ if and only if $f$
(the original polynomial mapping) is such an injection.
\end{theorem}

\noindent
{\it Ronen Peretz \\
Department of Mathematics \\ Ben Gurion University of the Negev \\
Beer-Sheva , 84105 \\ Israel \\ E-mail: ronenp@bgu.ac.il} \\ 
 
\end{document}